\documentclass[11pt]{amsart}

\usepackage{amsfonts,amssymb,amscd,amstext}
\usepackage{graphicx}
\usepackage[dvips]{epsfig} 
\usepackage{quoting}
\usepackage{hyperref} 

\usepackage{fancyhdr}
\input xy
\xyoption{all}

\usepackage{enumerate}
\usepackage{titlesec}
\usepackage{mathrsfs}

\titleformat{\section}
{\filcenter\bfseries\large} {\thesection{.}}{0.2cm}{}
\titleformat{\subsection}[runin]
{\bfseries} {\thesubsection{.}}{0.15cm}{}[.]
\titleformat{\subsubsection}[runin]
{\em}{\thesubsubsection{.}}{0.15cm}{}[.]

\usepackage[up,bf]{caption}

\newtheorem{theorem}{Theorem}[section]

\newtheorem{lemma}[theorem]{Lemma}

\newtheorem{corollary}[theorem]{Corollary}

\theoremstyle{definition}
\newtheorem{definition}[theorem]{Definition}
\newtheorem{remark}[theorem]{Remark}

\numberwithin{equation}{section}
\numberwithin{figure}{section}

\newcommand\Acal{\mathcal{A}}

\newcommand\Ccal{\mathcal{C}}

\newcommand\Ocal{\mathcal{O}}

\newcommand\C{\mathbb{C}}

\renewcommand\S{\mathbb{S}}

\renewcommand\c{\mathbb{C}}

\renewcommand\d{\mathbb D}

\newcommand\n{\mathbb{N}}
\renewcommand\r{\mathbb{R}}
\newcommand\s{\mathbb{S}}

\newcommand\hgot{\mathfrak{h}}

\newcommand\fgot{\mathfrak{f}}

\newcommand\bb{\mathrm{b}}

\usepackage{color}

\begin{document}


\fancyhead[LO]{Harmonic embeddings of open Riemann surfaces}
\fancyhead[RE]{A.\ Alarc\'on and F.J. L\'opez}
\fancyhead[RO,LE]{\thepage}

\thispagestyle{empty}


\begin{center}
{\bf\LARGE Proper harmonic embeddings\\ \smallskip of open Riemann surfaces into $\r^4$}

\bigskip

%
%
{\large\bf Antonio Alarc\'on\quad and\quad Francisco J. L\'opez}
\end{center}

%
%
\bigskip

\begin{quoting}[leftmargin={7mm}]
{\small
\noindent {\bf Abstract}\hspace*{1mm}
We prove that every open Riemann surface admits a proper embedding into $\r^4$ by harmonic functions. This reduces by one the previously known embedding dimension in this framework, dating back to a theorem by Greene and Wu from 1975.


\noindent{\bf Keywords}\hspace*{1mm} 
Riemann surface, harmonic function, proper embedding.


\noindent{\bf Mathematics Subject Classification (2020)}\hspace*{0.1cm} 
30F15, 
32Q40, 
53A05. 
}
\end{quoting}


\section{Introduction and main results}\label{sec:intro}

\noindent 
A real-valued function $h$ on a Riemannian manifold is called {\em harmonic} if it is a critical point of the energy functional; equivalently, $\Delta h=0$ where $\Delta$ is the Laplace operator associated to the Riemannian metric on the manifold. 

Greene and Wu proved in 1975 that every open (connected) Riemannian manifold of dimension $n\ge 2$ admits a proper embedding into $\r^{2n+1}$ by harmonic functions \cite{GreeneWu1975AIF}. 
The proof relies on approximation theory for this kind of functions and a general position argument. It is an open question whether $2n+1$ is the smallest embedding dimension in this framework.
In the case of $n=2$, harmonicity of a function on a Riemannian surface depends only on the conformal class of the Riemannian metric of the surface but not on the precise metric itself. Thus, for orientable surfaces, the result by Greene and Wu can be reformulated as follows:
Every open Riemann surface carries a proper
harmonic embedding into $\r^5$.
In this paper we improve this statement by reducing the embedding dimension by one.
%
%
\begin{theorem}\label{th:intro}
Every open Riemann surface admits a proper harmonic embedding into $\r^4$.
\end{theorem} 
Immersed surfaces in $\r^5$ are generically embedded, while transverse self-intersections (generically, double points) of immersed surfaces in $\r^4$ are stable under small deformations. This is the main reason why reducing the embedding dimension from $5$ to $4$ is a challenging undertaking. Lacking any kind of transversality or general position argument that help to ensure the embeddedness, our construction relies on an inductive procedure which never produces self-intersections, and so we need not get rid of them.  

The method by Greene and Wu \cite{GreeneWu1975AIF} consists first of constructing a proper harmonic map $f\colon X^n\to\r^{n+1}$ from an open $n$-dimensional Riemannian manifold $X^n$, as well as an injective harmonic immersion $g\colon X^n\to\r^{2n+1}$. Then, using a Whitney-type projection technique (see \cite{Whitney1937}), they extract from $(f,g)\colon X^n\to\r^{3n+2}$ a proper harmonic embedding $X^n\to\r^{2n+1}$.
For $n=2$, it has been known only recently that every open Riemann surface $M$ carries a proper harmonic map into $\r^2$; see \cite[Corollary 1.1]{AlarconLopez2012JDG} or \cite[Theorem 5.6]{AndristWold2014AIF}. This gave a complete solution to a question posed by Schoen and Yau in 1985 \cite[p.\ 18]{SchoenYau1997CIP}; see \cite[p.\ 175]{AlarconForstnericLopez2021Book} for further discussion and references. However, starting with a {\em proper} harmonic map $M\to\r^2$ (instead of into $\r^3$) does not lead to any dimensional reduction in the construction by Greene and Wu in \cite{GreeneWu1975AIF}. 
Instead of that, we shall begin with an {\em almost proper} holomorphic function on $M$, being in addition {\em simple} in the sense that its critical points are all of order one and the function is injective on its critical locus.\footnote{A continuous map $f\colon X\to Y$ between topological spaces is said to be {\em proper} if  $f^{-1}(K)$ is compact for every compact set $K\subset Y$, while it is called {\em almost proper} if every connected component of $f^{-1}(K)$ is compact for every such $K$.} Our method is therefore mixed: holomorphic -- harmonic. Besides this, our proof is completely different from that in \cite{GreeneWu1975AIF}. 

Here is our main result.
%
%
\begin{theorem}\label{th:h}
Let $M$ be an open Riemann surface. For any simple almost proper holomorphic function $z\colon M\to\c$ there is a harmonic map $h\colon M\to\r^2$ such that
$(z,h)\colon M\to \c\times\r^2\equiv \r^4$ 
is a proper embedding. 

Furthermore, the harmonic map $h$ can be chosen to be the real part of a holomorphic map $M\to\c^2$.
\end{theorem}
Admitting a proper holomorphic function is a strong restriction on an open Riemann surface. In particular, it implies that the surface is parabolic in the sense of Ahlfors-Nevanlinna \cite{Nevanlinna1941,Ahlfors1947}; equivalently, it does not carry nonconstant negative subharmonic functions (see \cite[Section IV.6]{AhlforsSario1960}). However, as shown by Bishop back in 1961 \cite{Bishop1961AJM}, every open Riemann surface $M$ admits an almost proper holomorphic function. In fact, the set of all such functions is residual (and hence dense) in the space of all holomorphic functions on $M$; see \cite[Chapter 7]{GunningRossi2009}. Since being simple is a generic condition, a standard transversality argument then guarantees that every open Riemann surface admits simple almost proper holomorphic functions, and hence Theorem \ref{th:h} trivially implies Theorem \ref{th:intro}.

We prove Theorem \ref{th:h} in Section \ref{sec:proofTh-h} via the main lemma (Lemma \ref{le:fun}), whose proof is deferred to Section \ref{sec:proof-lemma}. The proof consists in constructing a holomorphic map $M\to\C^2$ whose real part $h\colon M\to\r^2$ satisfies the following conditions. Recall that proper injective immersions $M\to\r^4$ are embeddings. 
\begin{enumerate}[11]
\item[$\bullet$] $h$ is a local diffeomorphism around every critical point of $z$. This implies that $(z,h)\colon M\to\r^4$ is an immersion.
\smallskip
\item[$\bullet$] $h$ separates the fibers of $z$ in the sense that $h|_{z^{-1}(\zeta)}$ is injective for all $\zeta\in\c$. This implies that $(z,h)$ is an injective map on $M$.
\smallskip
\item[$\bullet$] There are a divergent sequence $0<r_1<r_2<\cdots$ of regular values of $|z|$ and an exhaustion $K_1\Subset K_2\Subset\cdots \Subset\bigcup_{j\ge 1}K_j=M$ of $M$ by smoothly bounded compact connected regions such that $K_j$ is a component of $z^{-1}(r_j\overline\d)$ for all $j\ge 1$ (sequences with this property exist by almost properness of $z$), and we have $|h|>r_{j-1}$ everywhere on $z^{-1}(r_{j-1}\overline\d) \cap K_j\setminus K_{j-1}$ for all $j\ge 2$. Here $\d\subset \c$ denotes the open unit disc. This implies that $(z,h)\colon M\to \r^4$ is a proper map.
\end{enumerate}

We obtain the map $h$ as a uniform limit of a sequence of harmonic maps $h_j\colon K_j\to\r^2$ enjoying conditions analogous to the mentioned ones, but only in the compact set $K_j$. In particular, we shall make sure that every map $(z|_{K_j},h_j)\colon K_j\to\c\times\r^2$ in the sequence is injective. Roughly speaking, for constructing $h_j$ from $h_{j-1}$ we split $K_j\setminus \mathring K_{j-1}$ in a rather sophisticated sort of puzzle and carefully define $h_j$ in each piece, by using Runge-Mergelyan-type approximation with interpolation (see Theorem \ref{th:Arakelyan}) and exploiting the fact that $z|_{\mathring K_j}\colon \mathring K_j\to r_j\d$ is a simple branched covering. We construct the first and the second component functions of the harmonic map $h_j$ in turn, the second one strongly depending on the first one.

\begin{remark}\label{rem:n=3}
Theorem \ref{th:intro} proves Conjecture 3.10.7 in \cite{AlarconForstnericLopez2021Book}. It remains an open question whether $4$ is the smallest embedding dimension for open Riemann surfaces by harmonic functions.  We expect that it is. In fact, we conjecture that there exists no proper harmonic embedding $\d\to\r^3$. 
Proving this conjecture would provide an extension of the classical theorem by Heinz from 1952 stating that there exists no harmonic diffeomorphism $\d\to\r^2$ \cite{Heinz1952}.
\end{remark}

\begin{remark}\label{rem:nonorientable}
Theorem \ref{th:intro} implies that {\em every open orientable Riemannian surface admits a proper harmonic embedding into $\r^4$}. The statement in the theorem makes sense for nonorientable conformal surfaces as well. However, our proof does not seem to adapt to this case, and hence the question whether every open nonorientable Riemannian surface properly harmonically embeds in $\r^4$ remains open. 
\end{remark}

\begin{remark} \label{rem:Narasimhan}
Theorem \ref{th:h} shows that every open Riemann surface $M$ admits a proper harmonic embedding into $\r^4$ of the form $(z,h)\colon M\to\c\times\r^2$ with a holomorphic function $z$ on $M$. We do not know whether the function $h\colon M\to\r^2\equiv\c$ can be chosen holomorphic. Our results are thus connected with the classical Forster-Bell-Narasimhan Conjecture asking whether every open Riemann surface embeds as a smooth closed complex curve in $\c^2$ \cite{Forster1970CMH,BellNarasimhan1990EMS}. This seems one of the most difficult open problems in complex geometry. Only few open Riemann surfaces are known to carry a proper holomorphic embedding into $\c^2$; we refer to the survey in \cite[Sections 9.10-911]{Forstneric2017} 
for a history of this question. The results in the present paper provide some extra support for this longstanding conjecture to hold true.

Passing to higher dimensions, every Stein manifold $Z^k$ of complex dimension $k\ge 2$ is known to admit a proper holomorphic embedding into $\c^q$ for $q=[\frac{3k}2]+1$. This was proved by Eliashberg and Gromov \cite{EliashbergGromov1992AM} for even $k$ (their method of proof requires one more dimension, $\frac{3k+1}2+1$, for odd $k$), and by Sch\"urmann \cite{Schurmann1997MA} for odd $k$. The proper embedding dimension $[\frac{3k}2]+1$ is optimal for all dimensions $k\ge 2$ due to purely topological reasons, as Forster had previously shown by simple examples in \cite{Forster1970CMH}. Similar to ours, the proof by Sch\"urmann consists of starting with a generic almost proper holomorphic map $f\colon Z^k\to\C^k$ and finding a holomorphic map $g\colon Z^k\to \C^{q-k}$ such that the pair $(f,g)\colon  Z^k\to \C^q$ is a proper embedding.  Such a map $g$ is obtained by a suitable application of the Oka principle which cannot be used when $k=1$, and hence the proof  breaks down for open Riemann surfaces. (Eliashberg and Gromov began with a proper holomorphic map $Z^k\to\C^{k+1}$; this is why they lost one dimension compared to the optimal bound for odd $k$. We refer to \cite[Sections  9.3 and 9.4]{Forstneric2017} for the background information about this embedding problem and a full exposition of the proofs.)
We do not know whether our proof of Theorem \ref{th:h} could be simplified by implementing a suitable  homotopy principle for harmonic maps, which would allow to construct the two components of the  harmonic map $h$  simultaneously.
\end{remark}


\section{Preliminaries}\label{sec:prelim}

\noindent 
We let $\n=\{1,2,3,\ldots\}$ and denote by $|\cdot|$ the Euclidean norm in $\r^n$ for any $n\in\n$. For a set $A$ in a topological space $X$, we denote by $\overline A$ and $\mathring A$ the topological closure and interior of $A$ in $X$, respectively. For a subset $B\subset X$ we write $A\Subset B$ when $\overline A\subset\mathring B$. 
We use the notation $\Ccal^0(X,\r^n)$ or $\Ccal^0(X,\r)^n$ for the space of all continuous maps $X\to \r^n$, and set
\[
	\|f\|_{X}=\sup\{|f(x)|\colon x \in X\}, \quad f\in \Ccal^0(X,\r^n).
\]

An {\em open} surface is a non-compact surface with no boundary. In this paper surfaces are assumed to be connected unless the contrary is indicated. Given a topological surface, $M$, we denote by $\bb  M$  the $1$-dimensional topological manifold determined by its boundary points. Open connected subsets of $M\setminus \bb  M$ are called {\em domains}, while closed topological subspaces of $M$ whose all connected components are  surfaces with boundary  are said to be   {\em regions}. 
A region $W$ of  a topological surface $M$ (possibly with boundary)  is said to be a
{\em tubular neighborhood}  of a closed subset $F$ of $M$  if  $F\subset \mathring W$ and $F$ is a strong deformation retract of $W$.  

Let $M$ be an open possibly disconnected Riemann surface.  

A function $f\colon X\to\c$ on a subset $X\subset M$ is said to be {\em holomorphic} if it extends holomorphically to an unspecified open neighborhood of $X$ in $M$. As it is customary, we denote by $\Ocal(X)$ the space of all such functions. If $f\in\Ocal(M)$, then a point $p\in M$ is said to be a {\em critical point} or a {\em branch point} of $f$ if $df_p=0$; the set
\[
	{\rm Crit}(f)=\{p\in M\colon df_p=0\}
\]
of all such points is called the {\em critical set} or {\em critical locus} of $f$. We call $f({\rm Crit}(f))\subset \c$ as the set of {\em critical values} of $f$. A critical point $p$ of $f$ is said to be {\em simple} if the zero of $df$ at $p$ is of order $1$. 
%
%
\begin{definition}\label{def:simple-f}
A holomorphic function $f$ on an open Riemann surface $M$ is called {\em simple} if its critical points are all simple and $f$ is injective on ${\rm Crit}(f)$. If $X\subset M$, then we say that a function in $\Ocal(X)$ is {\em simple} if it is simple on an unspecified neighborhood of $X$ in $M$.
\end{definition}

Let $\Sigma\subset \c$ be a domain and assume that we have a holomorphic {\em branched covering}
$z\colon M\to \Sigma$. 
This means that every point $\zeta\in \Sigma$ admits an open disc neighborhood $\zeta\in D_\zeta\subset \Sigma$  such that each connected component $U$ of $z^{-1}(D_\zeta)$ contains a single point of $z^{-1}(\zeta)$ and 
\[
	z|_{U\setminus z^{-1}(\zeta)}\colon 
	U\setminus z^{-1}(\zeta)\to D_\zeta\setminus\{\zeta\}
\] 
is a finite  (unbranched) topological covering. If $A=z({\rm Crit}(z))$ and $B=z^{-1}(A)$, then $B$ is a closed discrete subset of $M$ and 
\[
	z|_{M\setminus B}\colon M\setminus B\to \Sigma\setminus A
\] 
 is an (unbranched) covering map; we denote by ${\rm d}_z\in \n\cup \{\infty\}$ its {\em degree}. 
The holomorphic branched covering $z\colon M\to \Sigma$ is said to be {\em finite} if ${\rm d}_z\in \n$; this is equivalent to $z\colon M\to \Sigma$ be a proper map. 

Let $r>0$. We denote by $\d\subset \c$ the open unit disc, $r\d=\{\zeta\in\c\colon |\zeta|<r\}$, and  $r\s^1=\bb (r\overline\d)$.  If $z\in \Ocal(K)$ for an {\em analytic} compact region $K\subset M$ such that $z|_{\mathring K}\colon \mathring K\to r\d$ is a branched covering, then the function $z$ has no critical points in  $\bb K$. In this case, we have that ${\rm Crit}(z|_{\mathring K})$ is finite and the branched covering $z|_{\mathring K}\colon  \mathring K\to r\d$ is finite. 
On the other hand, if $z\in \Ocal(M)$ is almost proper, then for any component $\Omega$ of $z^{-1}(r\d)$, the map $z|_\Omega\colon \Omega\to r\d$ is a finite branched covering; in particular, $z$ is surjective. These two elementary facts will be recurrently used throughout the paper. 

A compact subset $K\subset M$  is said to be {\em Runge} (also {\em holomorphically convex} or {\em $\mathcal{O}(M)$-convex}) if it has no holes: $M\setminus K$ has no relatively compact connected components. For instance, if $z\in \Ocal(M)$ and $K\subset M$ is a compact region such that  $z|_{\mathring K}\colon \mathring K\to r\d$ is a finite branched covering for some $r>0$, then $K$ is Runge in $M$. 

For a subset $X\subset M$, we denote by $\Re\Ocal(X)$ the space of functions $h\colon X\to\r$ such that $h=\Re(f)$ for some function $f\in\Ocal(X)$. Likewise, we set $\Acal(X)=\Ccal^0(X,\c)\cap \Ocal(\mathring X)$
and denote by $\Re\Acal(X)$ the space of functions $h\colon X\to\r$ such that $h=\Re(f)$ for some function $f\in\Acal(X)$.
We shall say that a function $h\in \Re\Acal (X)$ {\em can be uniformly approximated on $X$ by functions in $\Re\Ocal (M)$} if 
 \[
 	\inf\big\{\|v-h\|_{X}\colon v\in \Re\Ocal(M)\big\}=0.
\]
In this framework, the classical Runge-Mergelyan approximation theorem with interpolation on open Riemann surfaces (see Bishop \cite{Bishop1958PJM} or e.g. \cite[Theorem 1.12.1]{AlarconForstnericLopez2021Book} or \cite[Theorem 5]{FornaesForstnericWold2020}) reads as follows. Recall that a harmonic function on a Riemann surface is locally the same thing as the real part of a holomorphic function.
\begin{theorem}\label{th:Arakelyan}
If $K$ is a Runge compact set in an open Riemann surface $M$ and $\Upsilon\subset \mathring K$ is a  finite set, then every  function $h\in\Re\Acal(K)$  can be  uniformly approximated   on $K$ by  functions $\tilde h\in\Re\Ocal(M)$ such that $\tilde h-h$ vanishes to any given order $k\ge 0$ at every point in $\Upsilon$.
\end{theorem}


\section{Proof of Theorem \ref{th:h}}\label{sec:proofTh-h}

\noindent 
%
%
The proof of Theorem \ref{th:h} follows from an inductive application of the following lemma and its subsequent corollary. We shall defer the proof of the lemma for Section \ref{sec:proof-lemma}.
\begin{lemma} \label{le:fun}
Let $M$ be an open Riemann surface, let $z\colon M\to \c$ be a simple holomorphic function (Definition \ref{def:simple-f}), and assume that $0<r<s$ are noncritical values of $|z|$ such that
$S=z^{-1}(s \overline \d)$ 
is compact and connected. Define
$R=z^{-1}(r \overline \d)\Subset S$ and let $h=(h_1,h_2)\in\Re\Ocal(R)^2$ such that:
\begin{enumerate}[{\rm (I)}]
\item  $(z|_R,h)\colon  R \to \c\times \r^2$ is an embedding.
\smallskip
\item $\hgot_1\cup \hgot_2$ is finite and $\hgot_1\cap \hgot_2=\varnothing$, where $\hgot_j=\{\zeta\in r\s^1\colon h_j|_{z^{-1}(\zeta)}$ is not injective$\}$, $j=1,2$.
\end{enumerate}
If we are given a finite set $\hgot\subset s\s^1$, then for any $\epsilon>0$ there exists a map $\tilde h=(\tilde h_1,\tilde h_2)\in \Re\Ocal(S)^2$ satisfying the following conditions:
\begin{enumerate}[{\rm (i)}]
\item $(z|_{S},\tilde h)\colon  S  \to  \c\times \r^2$ is an embedding.
\smallskip   
\item $\|\tilde h-h\|_R< \epsilon$ and $\tilde h-h$ vanishes to order $1$ at every point in ${\rm Crit}(z)\cap R$.\hspace*{-2mm}
\smallskip
\item $\tilde \hgot_1\cup \tilde \hgot_2$ is finite, $\tilde \hgot_1\cap \tilde \hgot_2=\varnothing$, and $(\tilde \hgot_1\cup\tilde \hgot_2)\cap\hgot=\varnothing$, where $\tilde \hgot_j=\{\zeta\in s\s^1\colon \tilde h_j|_{z^{-1}(\zeta)}$ is not injective$\}$, $j=1,2$.
\end{enumerate}
\end{lemma}
We emphasize that both $S$ and $R$ are compact analytic regions in $M$; $R$ need not be connected. Moreover, $z|_{\mathring S}\colon \mathring S   \to s   \d$ is a simple branched covering.
Conditions {\rm (I)} and {\rm (II)} hold for any map $h\in\Re\Ocal(R)^2$ provided that   $z|_R$ is a biholomorphism. 
%
%
\begin{corollary} \label{co:fun}
Let $M$, $z$, $s$, $S$, and $\hgot$ be as in Lemma \ref{le:fun}, and assume that $0\in\c$ is a noncritical value of $z$. Then, for any $\lambda>0$  there exists $h=(h_1,h_2)\in \Re\Ocal(S)^2$ satisfying the following conditions:
\begin{enumerate}[{\rm (i)}]
\item $(z|_{S}, h)\colon  S  \to  \c\times \r^2$ is an embedding.

\smallskip   
\item $\hgot_1\cup \hgot_2$ is finite, $\hgot_1\cap \hgot_2=\varnothing$, and $(\hgot_1\cup \hgot_2)\cap\hgot=\varnothing$, where $\hgot_j=\{\zeta\in s\s^1\colon h_j|_{z^{-1}(\zeta)}$ is not injective$\}$, $j=1,2$.

\smallskip
\item $|h_j(p)|>\lambda$ for all $p\in S$, $j=1,2$.
\end{enumerate}
\end{corollary}
\begin{proof}[Proof of Corollary \ref{co:fun} assuming Lemma \ref{le:fun}]
Let $0<r<s$ be so small that $z$ has no critical values in $r\overline\d$ (use that $0$ is a noncritical value of $z$ and recall that $S=z^{-1}(s\overline\d)$ is compact). Observe that $R:=z^{-1}(r \overline \d)$ is a union of $m={\rm d}_{z|_S}\ge 1$ (the degree of $z|_S$) pairwise disjoint closed discs, all of them biholomorphic to $r\overline\d$ via $z$.
Choose any map $f=(f_1,f_2)\in \Re\Ocal(R)^2$ such that $(z|_R,f)\colon  R \to \c\times \r^2$ is an embedding and $f_j|_{z^{-1}(\zeta)}$ is injective for all $\zeta\in r\s^1$ and  $j\in\{1,2\}$.
Such a map $f$ trivially exists; it can even be chosen locally constant. Lemma \ref{le:fun} applied to these data furnishes us with a map $  h\in\Re\Ocal(S)^2$ satisfying {\rm (i)} and {\rm (ii)}. Since $S$ is compact and $h$ is continuous, there is a constant $c\in\r^2$ such that the map $c+ h$ meets also {\rm (iii)}.
\end{proof}
%
%

Granted Lemma \ref{le:fun}, the proof of Theorem \ref{th:h} is completed as follows.

Let $M$ be an open Riemann surface and assume that $z\colon M\to \c$ is a simple almost proper holomorphic function. Up to replacing $z$ by $z-\zeta_0$ for a noncritical value $\zeta_0\in\c$ of $z$, we may and shall assume that $0\in\c$ is a noncritical value of $z$.
Let $0<r_1<r_2<\cdots$ be a divergent sequence of noncritical values of $|z|$. Since $z$ is almost proper, an elementary topological argument furnishes us with an exhaustion 
\[
	K_1\Subset K_2\Subset \cdots 
	\Subset  \bigcup_{j\in \n} K_j=M
\]
of $M$ by compact connected analytic regions such that $K_j$ is a component of $z^{-1}(r_j\overline\d)$ for all $j\in\n$. Note that ${\rm Crit}(z)\cap \bb K_j=\varnothing$ and $z|_{\mathring K_j}\colon  \mathring K_j\to r_j   \d$ is a simple branched covering for all $j\in \n$. Set $K_0=\varnothing$ and $r_0=0$. 

Let ${\rm dist}\colon M\times M\to [0,+\infty)$ be the distance associated to any fixed Riemannian metric on $M$. We claim that there
is a sequence of maps
\begin{equation}\label{eq:virgen-santa}
	  h_i=(h_{i,1},h_{i,2})\in\Re\Ocal(K_i)^2, \quad i\in \n,
\end{equation}
satisfying the following conditions for all $i\in\n$:
\begin{enumerate}[{\rm (a$_{i}$)}]
\item $(z|_{K_i}, h_i)\colon  K_i  \to  \c\times \r^2$ is an embedding.

\smallskip
\item $\hgot_{i,1}\cup \hgot_{i,2}$ is finite and $\hgot_{i,1}\cap \hgot_{i,2}=\varnothing$, where $\hgot_{i,j}=\{\zeta\in r_i\s^1\colon h_{i,j}|_{z^{-1}(\zeta)}$ is not injective$\}$, $j=1,2$.

\smallskip
\item $|h_{i,j}(p)|>r_{l-1}$ for all $p\in z^{-1}(r_{l-1}\overline\d)\cap K_l\setminus K_{l-1}$, $l=1,\ldots,i$, and $j=1,2$.

\smallskip   
\item $h_i-h_{i-1}$ vanishes to order $1$ at every point in ${\rm Crit}(z)\cap K_{i-1}$ and $\|h_i-h_{i-1}\|_{K_{i-1}}< \min\{1/2^i\,,\, \delta_i\}$, where
\[
	\delta_i= \frac1{2i^2}\inf \Big\{ |(z,h_{i-1})(p)-(z,h_{i-1})(q)|\colon p,q\in K_{i-1},\; {\rm dist}(p,q)>\frac1{i} \Big\}.
\] 
\end{enumerate}
Indeed, we proceed by induction. The base case is ensured by Corollary \ref{co:fun} applied with $M$ replaced by the interior of a suitable  tubular neighborhood of $K_1$ to the data $z$, $s=r_1$, $S=K_1$, $\hgot=\varnothing$, and any $\lambda>0$. This gives a map $h_1=(h_{1,1},h_{1,2})\in \Re\Ocal(K_1)^2$ satisfying {\rm (a$_1$)}, {\rm (b$_1$)}, and {\rm (c$_1$)}; note that {\rm (d$_1$)} is empty. For the inductive step, fix $i\in\n$, suppose that we have $h_i=(h_{i,1},h_{i,2})\in \Re\Ocal(K_i)^2$ satisfying conditions {\rm (a$_i$)}--{\rm (d$_i$)}, and let us construct $h_{i+1}$.
Recall that $K_{i+1}$ is a compact connected analytic region and  $z|_{\mathring K_{i+1}}\colon\mathring K_{i+1}\to r_{i+1}\d$ is a simple branched covering. We define 
\[
	K_{i}'= z^{-1}(r_{i}\overline \d)\cap K_{i+1} 
\]
and write $S_1,\ldots,S_k$ for its connected components, where $S_1=K_i$. Note that 
$z|_{S_m}\colon S_m\to r_i\overline \d$ is a  simple branched covering, $m=1,\ldots,k$. Define $L_m=\bigcup_{l=1}^m S_l$ for $m\in\{1,\ldots,k\}$, and let $f_1=h_i$.
We claim that there is a sequence of maps $f_m=(f_{m,1},f_{m,2})\in\Re\Ocal(L_m)^2$, $m=1,\ldots,k$, such that:
\begin{enumerate}[{\rm (A{$_m$})}]
\item $(z|_{L_m}, f_m)\colon  L_m  \to  \c\times \r^2$ is an embedding.

\smallskip
\item $\fgot_{m,1}\cup \fgot_{m,2}$ is finite and $\fgot_{m,1}\cap \fgot_{m,2}=\varnothing$, where $\fgot_{m,j}=\{\zeta\in r_i\s^1\colon f_{m,j}|_{z^{-1}(\zeta)}$ is not injective$\}$, $j=1,2$.

\smallskip
\item $|f_{m,j}(p)|>r_i$ for all $p\in L_m\setminus S_1=\bigcup_{l=2}^m S_l$, $j=1,2$.

\smallskip
\item $f_m|_{L_{m-1}}=f_{m-1}$.
\end{enumerate}
The basis of the induction is given by $f_1=h_i$; note that {\rm (A$_1$)}={\rm (a$_i$)} and {\rm (B$_1$)}={\rm (b$_i$)} while {\rm (C$_1$)} and {\rm (D$_1$)} are empty. Assume that we have a map $f_m\in\Re\Ocal(L_m)^2$ for some $m\in\{1,\ldots,k-1\}$ satisfying  {\rm (A$_m$)}--{\rm (D$_m$)}. Let $g\in \Re\Ocal(S_{m+1})^2$ be provided by Corollary \ref{co:fun} applied with $M$ replaced by the interior of a suitable tubular neighborhood of $S_{m+1}$ to the data $z$, $s=r_i$, $S=S_{m+1}$, $\hgot=\fgot_{m,1}\cup \fgot_{m,2}$, and 
\begin{equation}\label{eq:lambda-grande}
	\lambda>\max\{r_i,\|f_m\|_{L_m}\}. 
\end{equation}
Extend $f_m$ to $L_{m+1}=L_m\cup S_{m+1}$ as a map $f_{m+1}\in \Re\Ocal(L_{m+1})^2$ by defining $f_{m+1}|_{S_{m+1}}=g$. It is clear that $f_{m+1}$ satisfies conditions {\rm (A$_{m+1}$)}--{\rm (D$_{m+1}$)}. In particular, {\rm (B$_{m+1}$)} is guaranteed by {\rm (B$_m$)}, condition {\rm (ii)} in the corollary, and inequality \eqref{eq:lambda-grande}. This concludes the construction of the maps $f_1,\ldots,f_k$.

Conditions {\rm (A$_k$)} and {\rm (B$_k$)} enable us to apply Lemma \ref{le:fun} with $M$ replaced by the interior of a suitable  tubular neighborhood of $K_{i+1}$ to the data
\[
	z,\; r=r_i,\; s=r_{i+1},\; S=K_{i+1},\; R= K_{i}'=L_k,\; h=f_k\in\Re\Ocal(K_{i}')^2,\; \hgot=\varnothing,
\]
and $0<\epsilon<\min\{2^{-(i+1)},\delta_{i+1}\}$. Choosing $\epsilon>0$ sufficiently small, this furnishes us with a map $h_{i+1}\in \Re\Ocal(K_{i+1})^2$ satisfying {\rm (a$_{i+1}$)}--{\rm (d$_{i+1}$)}, thereby closing the induction. In particular, {\rm (c$_{i+1}$)} is ensured by {\rm (c$_i$)}, {\rm (C$_k$)}, and the fact that $h_{i+1}$ is $\epsilon$-close to $f_k$ on $K_i'$. Note that $\delta_{i+1}>0$ by condition {\rm (a$_i$)}. This completes the inductive construction of the maps in \eqref{eq:virgen-santa}.

By properties {\rm (a$_i$)} and {\rm (d$_i$)}, $i\in \n$, the sequence $\{h_i\}_{i\in \n}$ converges uniformly on compact subsets of $M$ to a map $h \in\Re\Ocal(M)^2$ such that 
\begin{equation}\label{eq:zhimmersion}
	(z,h)\colon M=\bigcup_{i\in \n}K_i\to \r^4\quad \text{is an immersion}.
\end{equation}
We claim that $h$ satisfies the conclusion of the theorem. Indeed, let $p,q\in M$, $p\neq q$, and choose $i_0\in\n$ so large that $p,q \in K_{i_0}$ and ${\rm dist}(p,q)>1/i_0$. Conditions {\rm (d$_i$)} ensure that 
\[
	|(z,h_i)(p)-(z,h_i)(q)|\ge (1-\frac1{i^2})|(z,h_{i-1})(p)-(z,h_{i-1})(q)| \text{ for all $i>i_0$}, 
\]
and hence
\[
	|(z,h_{i_0+l})(p)-(z,h_{i_0+l})(q)|\ge |(z,h_{i_0})(p)-(z,h_{i_0})(q)|\cdot\prod_{i=i_0+1}^{i_0+l}\Big(1-\frac1{i^2}\Big), 
\]
for all $l\in\n$. Taking limits as $l\to\infty$ and using {\rm (a$_{i_0}$)} we obtain that $|(z,h)(p)-(z,h)(q)|\ge \frac12|(z,h_{i_0})(p)-(z,h_{i_0})(q)|>0$. This shows that  
\begin{equation}\label{eq:zhinjective}
	(z,h)\colon  M\to \r^4 \quad\text{is injective}.
\end{equation}
On the other hand, conditions {\rm (c$_j$)}, $j>i$, guarantee that 
\[
	|h(p)|\ge r_{i-1}\quad \text{for all 
	$p\in z^{-1}(r_{i-1}\overline\d)\cap K_i\setminus K_{i-1}$,}
\]
which implies that $|(z,h)(p)|\ge r_{i-1}$ for all $p\in K_i\setminus \mathring K_{i-1}$ and $i\in\n$. Since $\lim_{i\to\infty}r_i=+\infty$ and the sets $K_i$ are all compact, this ensures that $(z,h)\colon M=\bigcup_{i\in \n}K_i\to \r^2$ is a proper map.
This, \eqref{eq:zhimmersion}, and \eqref{eq:zhinjective} imply that $(z,h)\colon M\to \r^4$ is a proper embedding. 

This concludes the proof of Theorem \ref{th:h} under the assumption that Lemma \ref{le:fun} holds true. 
 
 \section{Proof of Lemma \ref{le:fun}}\label{sec:proof-lemma}
\noindent 
The proof consists of two main stages. In the first one we approximate $h_1$ on $R$ by a function $\tilde h_1\in\Re\Ocal(S)$ such that $(z,\tilde h_1)$ is injective on a certain subset of $S\setminus R$. We shall first describe a $z$-saturated Runge compact set $\hat K_1$ in $M$ with $R\subset \hat K_1\subset S$ (see \eqref{eq:K1}). We shall then suitably extend $h_1$ to $\hat K_1$ as a function of class $\Re\Acal(\hat K_1)$ and, finally, we shall make use of Theorem \ref{th:Arakelyan} to obtain the function $\tilde h_1$. In the second stage of the proof we approximate $h_2$ by a function $\tilde h_2\in\Re\Ocal(S)$ such that $(z,\tilde h_2)$ is injective on the complement of that set in $S\setminus R$, by using a similar procedure. So, each main stage in the proof consists of three different steps. The construction will ensure that $\tilde h=(\tilde h_1,\tilde h_2)$ satisfies the conclusion of the lemma.
 
We begin with some preparations. The assumptions in the lemma imply that ${\rm Crit}(z)\cap (\bb S\cup\bb R)=\varnothing$ and $z|_{\mathring S}\colon \mathring S\to s\d$ is a branched covering. Choose a finite set $\hgot\subset s\S^1$ and fix $0<\epsilon<1$.
Define 
\[
	B_0={\rm Crit}(z)\cap \mathring S\setminus R,\quad A=z(B_0),\quad \text{and}\quad B=z^{-1}(A).
\] 
These sets may be empty. Since $z$ is simple, the set $z^{-1}(a)\cap B_0$ consists of a single point for all $a\in A$. Since $S=z^{-1}(s\overline\d)$ and $R=z^{-1}(r\overline\d)$, we have
\begin{equation}\label{eq:B0B}
 	B_0\subset B\subset \mathring S\setminus R 
 	\quad \text{and}\quad  
	A \subset s\d\setminus r \overline \d.
\end{equation}

%
%
\subsubsection*{Step 1: The first Runge compact set} For $j=1,2$, define
\begin{equation}\label{eq:Cjde}
 	[ h_j ] =\{p\in R \colon \exists\, q\in R \setminus  \{p\} 
	\text{ such that }(z,h_j)(p)=(z,h_j)(q)\}.
\end{equation}
The set $[h_j]$ is real analytic by harmonicity of $h$, and it could be empty. We have that $z([h_j])\cap r\s^1=\hgot _j$, and hence condition {\rm (II)} in the statement of the lemma implies that $\dim_\r([h_j])\leq 1$ for $j=1,2$,
\begin{equation}\label{eq:zC1zC2}
	z([h_1])\cap z([h_2])\cap r \s^1=\varnothing,\quad \text{and}\quad
	([h_1]\cup[h_2])\cap \bb R\; \text{ is finite}.
\end{equation}
Moreover, up to replacing $r$ by a slightly larger number if necessary, we can assume in addition that
\begin{equation} \label{eq:Cjtr}
	\text{$[h_j]$ is transverse to $\bb  R$, $j=1,2$.}
\end{equation}
Fix  a finite set $\varnothing \neq A_0\subset r\s^1$ such that 
\begin{equation} \label{eq:A0}
z([h_2])\cap r\s^1=z([h_2] \cap \bb  R)\subset A_0\subset r \s^1\setminus z([h_1])
\end{equation}
and
 \begin{equation}\label{eq:A0zC1}
\text{every component of $r \s^1\setminus A_0$ contains at most one point of $z([h_1])$;}
\end{equation}
see \eqref{eq:zC1zC2} and \eqref{eq:Cjtr}. Note that $A\cap A_0=\varnothing$ by \eqref{eq:B0B}. Set
\[
	A^*= A\cup A_0\subset s\d\setminus r \d.
\]  
Take a family $\{\Gamma_a\colon a\in A^*\}$ of pairwise disjoint segments of the form 
\[
	\Gamma_a=[a,e_a]\subset s\overline\d\setminus (r \d\cup \hgot),\quad e_a\in s\s^1,
\] 
such that $\Gamma_a$ is transverse to $r\s^1$ for all $a\in A_0$. It turns out that $\Gamma_a\cap (r\s^1\cup s \s^1)= \{a, e_a\}$ for all $a\in A_0$, $\Gamma_a\cap (r\s^1\cup s \s^1)= \{e_a\}$ for all $a\in A$, and  $\Gamma_a$ is transverse to $s\s^1$  for all   $a\in A^*$. 
Call 
\[
	\Gamma_{A_0}=\bigcup_{a\in A_0} \Gamma_a,\quad \Gamma_A=\bigcup_{a\in A} \Gamma_a, \quad \text{and}\quad \Gamma=\Gamma_{A_0}\cup \Gamma_A.
\]
We may and shall choose $A_0$ and $\{e_a\colon a\in A^*\}$ so that 
every connected component of $s \overline \d\setminus (r \d \cup \Gamma_{A_0})$ contains at most one point of $A$ and at most one point of $\hgot$.
See Figure \ref{fig:A-A0}.
\begin{figure}[ht]
\begin{minipage}[b]{0.45\linewidth}
\centering
	\includegraphics[width=\textwidth]{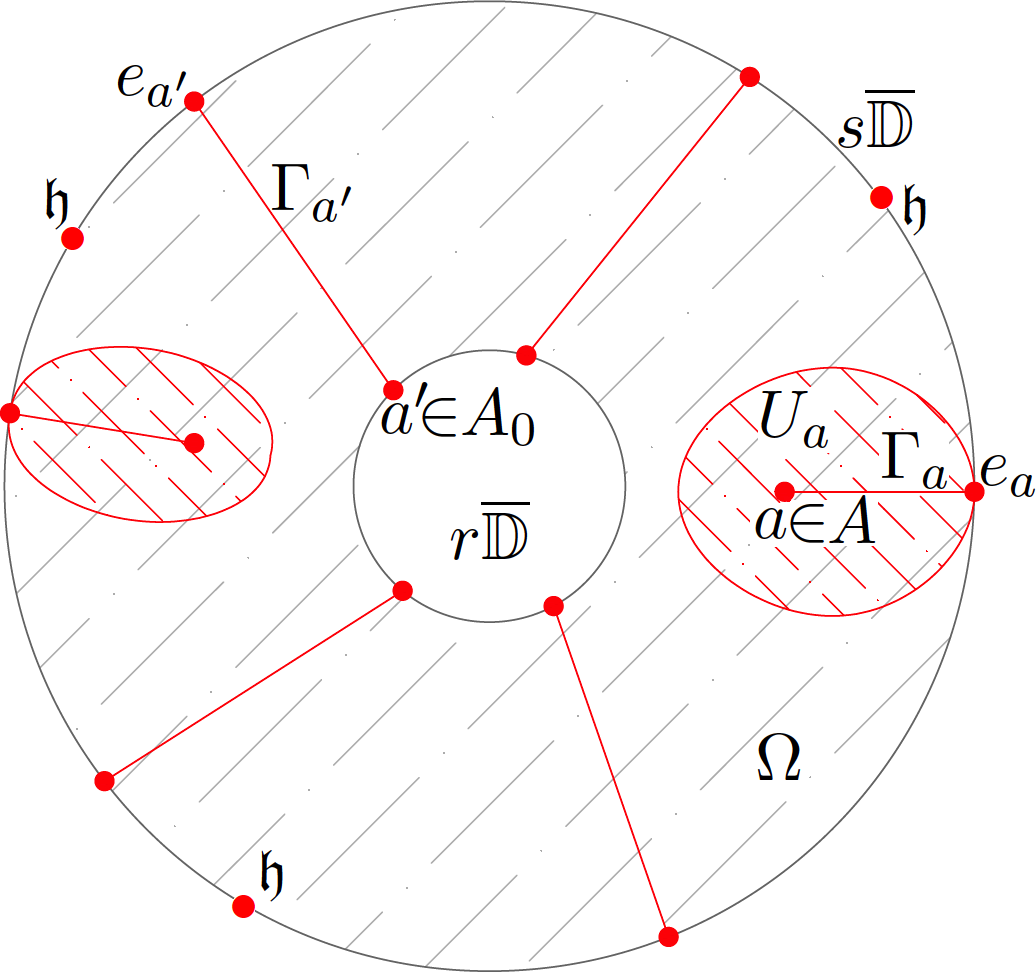}
 \end{minipage} 
   \hspace{5mm}
     \begin{minipage}[b]{0.4\textwidth}
 \centering
	\includegraphics[width=\textwidth]{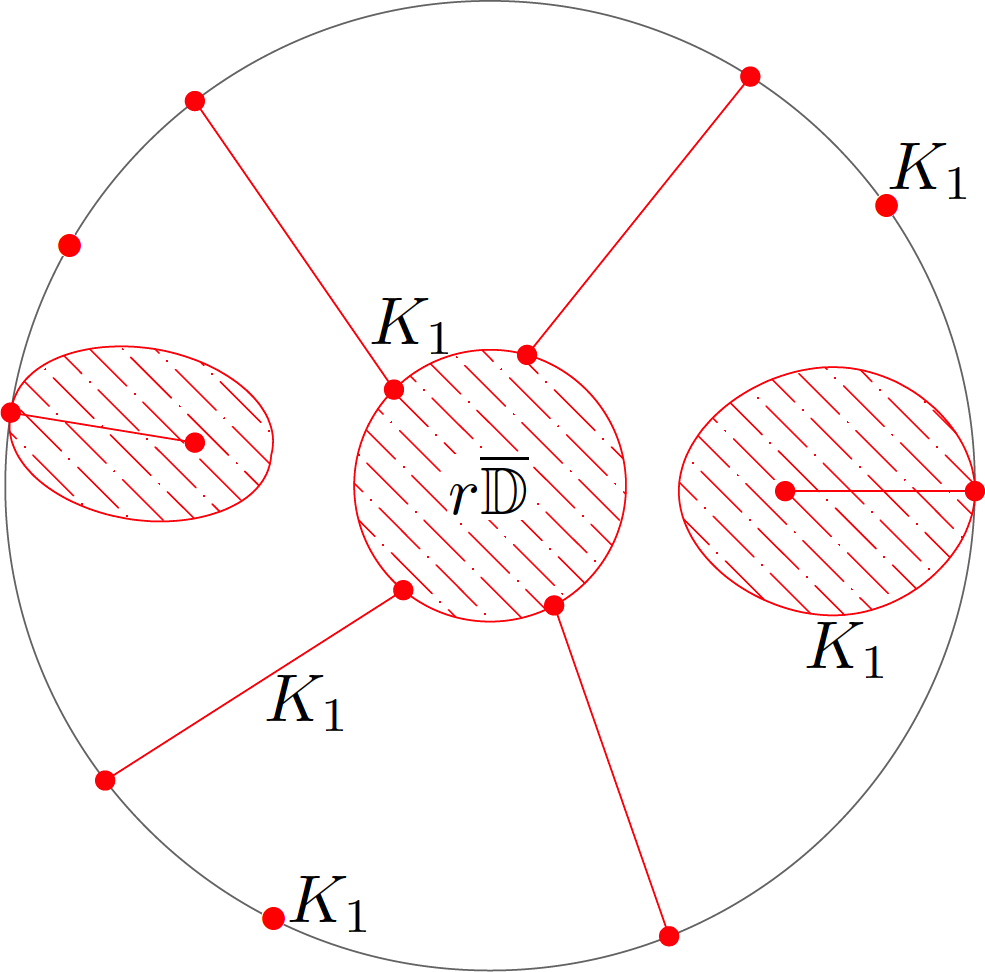}
 \end{minipage} 
	\caption{Left: The sets $\Gamma_a$ and $U_a$. Right: The set $K_1$.}\label{fig:A-A0}
\end{figure} 
Note that each connected component of 
$s\overline \d \setminus (r  \d \cup \Gamma)$ is  simply connected. 
Moreover, the restricted function
\[
	z|_{S\setminus (\mathring R\cup B)}\colon S\setminus (\mathring R\cup B)\to  s\overline \d \setminus (r  \d \cup A)
\]
(see \eqref{eq:B0B})
is an unbranched finite covering of degree ${\rm d}_{z|_{\mathring S}}\geq 1$.
Therefore, for any component $\Omega$ of  
$s\overline \d \setminus (r  \d \cup \Gamma)\subset s\overline \d \setminus (r  \d \cup A)$ the set $z^{-1}(\Omega)\subset S\setminus (\mathring R\cup B)$ has  ${\rm d}_{z|_{\mathring S}}$ connected components, and for any component $\hat \Omega$ of  $z^{-1}(\Omega)$ we have that
$z|_{\hat \Omega}\colon \hat \Omega \to \Omega$ is a biholomorphism.
Furthermore, \eqref{eq:zC1zC2} and \eqref{eq:A0} ensure that $\bb  R \cap z^{-1}(\Omega)\cap [h_2]=\varnothing$, and hence if $\hat \Omega$ and $\hat \Omega'$ are different components of  $z^{-1}(\Omega)$ then the function
\[
	h_2\circ (z|_{\bb  R \cap \hat \Omega'})^{-1}- h_2\circ (z|_{\bb  R \cap \hat \Omega})^{-1}\colon r \s^1\cap \Omega\to \r
\]
vanishes nowhere. Since $r \s^1\cap \Omega$ is connected, the binary relation given by
 \begin{equation}\label{eq:OmOm'}
 	\hat \Omega'< \hat \Omega \quad\text{if and only if}\quad 
	h_2\circ (z|_{\bb  R \cap \hat \Omega'})^{-1}< h_2\circ (z|_{\bb  R \cap \hat \Omega})^{-1}
 \end{equation}
 determines a total order in the set of components of $z^{-1}(\Omega)$. 
 
Next, we also choose a family of pairwise disjoint analytic closed discs  $\{U_a\subset \c\colon a\in A\}$ such that $U_a\subset s\overline \d\setminus (r \overline \d\cup \Gamma_{A_0})$, $U_a\cap s\s^1=\{e_a\}$, $\Gamma_a\setminus \{e_a\}\subset \mathring U_a$ (hence $\Gamma_a\subset U_a$), and 
\begin{equation}\label{eq:Ua}
	\text{$U_a$ is invariant under the reflection about the line containing $\Gamma_a$,}
\end{equation}
for all $a\in A$. See Figure \ref{fig:A-A0}. Set
\[
  	U_A=\bigcup_{a\in A} U_a\quad \text{and} 
	\quad \hat U_A=z^{-1}(U_A)\subset S.
\]
Likewise, for $a\in A^*$ we let $\hat\Gamma_a=z^{-1}(\Gamma_a)\subset S$,
\begin{equation}\label{eq:hatGammaA_0}
	\hat\Gamma_{A_0}=\bigcup_{a\in A_0} \hat\Gamma_a, \quad 
	\hat\Gamma_A=\bigcup_{a\in A} \hat\Gamma_a, \quad \text{and}\quad \hat\Gamma=\hat\Gamma_{A_0}\cup \hat\Gamma_A.
\end{equation}
Finally, define $\hat\hgot =z^{-1}(\hgot)\subset\bb S$,
\begin{equation}\label{eq:K1}
K_1=r \overline\d \cup \Gamma_{A_0}\cup U_A\cup \hgot, \; \text{ and } \; \hat K_1:=z^{-1}(K_1)=R\cup \hat \Gamma_{A_0}\cup \hat U_A\cup\hat\hgot .
\end{equation}
See Figure \ref{fig:A-A0}.
Note that $\hat K_1$ is a Runge compact set in $M$, $\hat U_A\cap\hat \hgot=\varnothing$, $(R\cup \hat \Gamma_{A_0})\cap (\hat U_A\cup\hat\hgot )=\varnothing$,  and $z(\hat K_1)=K_1$.  The set $\hat K_1$ induces a sort of topological puzzle on $S\setminus \mathring R$ which allows to describe the covering $z|_{S\setminus \mathring R}\colon S\setminus \mathring R\to s\overline\d\setminus r\d$ by means of topological surgery. We explain the details.

Fix $b\in B=z^{-1}(A)$ and write $a=z(b)\in A$. Denote by $\hat \Gamma^b\subset \hat \Gamma_a$ the component of $\hat\Gamma_A$ containing $b$ and by $\hat U^b$ the component of $\hat U_A$ containing $b$ (in its interior); the latter is a closed disc. 

Assume that $b\in B\setminus B_0$ is noncritical for $z$. Write $\hat \Omega^b$    for the   component of $S\setminus (\mathring R\cup \hat \Gamma)$ whose closure (a compact disc) contains $\hat U^b$.  
In this case  
\begin{equation}\label{eq:zUbnoB0}
	z|_{\hat U^b}\colon \hat U^b\to U_a \quad \text{is a biholomorphism}.
\end{equation}
Suppose on the contrary that $b\in B_0$ is critical for $z$. Since $z$ is simple, there are two components $\hat \Omega^b_-$ and $\hat \Omega^b_+$ of $S\setminus (\mathring R\cup \hat \Gamma)$ whose closures (compact discs) contain $\hat \Gamma^b$. In fact, $\hat \Gamma^b=\overline{\hat \Omega^b_-}\cap \overline{\hat \Omega^b_+}$.  Setting $\hat \Omega^b:=\hat \Omega^b_-\cup \hat \Omega^b_+$, we have that $z(\hat \Omega^b)=z(\hat \Omega^b_-)=z(\hat \Omega^b_+)$ is a component of $s\overline \d \setminus (r  \d \cup \Gamma)$ and $z|_{\hat \Omega^b \cup \hat \Gamma^b}\colon \hat \Omega^b\cup\hat \Gamma^b\to  z(\hat \Omega^b)\cup  \Gamma_a$ is a branched covering of degree $2$ with the only (interior) branch point $b$; observe that $\hat U^b\subset   \hat \Omega^b\cup\hat \Gamma^b$. Let 
 \begin{equation}\label{eq:J^b}
 J^b\colon   \hat \Omega^b\cup\hat \Gamma^b\to   \hat \Omega^b \cup\hat \Gamma^b 
 \end{equation}
  denote the only holomorphic involution with $z\circ J^b=z$, and notice that
$J^b(\hat U^b)=\hat U^b$,   $J^b(\hat \Gamma^b)=\hat \Gamma^b$, $J^b(\hat \Omega^b_+)=\hat \Omega^b_-$, and $J^b(b)=b$. We assume, as we may up to relabeling, that 
\begin{equation}\label{Omega-Omega+}
	\hat \Omega^b_-<\hat \Omega^b_+
\end{equation}
in the sense of \eqref{eq:OmOm'}. The endpoints $e^b_1,e^b_2$ of the arc $\hat \Gamma^b\subset \hat U^b$ lie in  $\bb  \hat U^b\cap \bb S$ and satisfy  $\{e^b_1,e^b_2\}=\hat \Gamma^b\cap \bb  \hat U^b$ and $z(e^b_1)=z(e^b_2)=e_a$. In this case 
\begin{equation}\label{eq:zUbsiB0}
	\left\{\begin{array}{l}
	z|_{\hat U^b}\colon \hat U^b\to U_a 
	\text{ is a branched covering of degree $2$}\smallskip
	\\
	\text{with the only branch point $b$.}
	\end{array}\right.
\end{equation}

%
%
\subsubsection*{Step 2: Extending $h_1$ to $\hat K_1$}
Our next task is to suitably extend $h_1$ to $\hat K_1$ as a function in $\Re\Acal(\hat K_1)$. In particular, we shall ensure that $(z,h_1)$ is injective on a certain large subset of $\hat K_1\setminus R$. Again, fix $b\in B=z^{-1}(A)$ and write $a=z(b)\in A$. First, we consider the only biholomorphism 
\[
	\varphi_a\colon U_a\to   \overline \d
\] 
such that $\varphi_a(a)=0$ and $\varphi_a(e_a)= 1$.
Condition \eqref{eq:Ua} ensures that $\varphi_a(\Gamma_a)=[0,1]$. 
If $b\in B_0$ is critical, then we define 
$\psi^b\colon \hat U^b\to \overline \d$ as the only biholomorphic map satisfying
\begin{equation}\label{eq:psiOm+-}
	 (\psi^b)^2=\varphi_{a}\circ z|_{\hat U^b} \quad \text{and}\quad 
	 \psi^b( \hat U^b\cap \hat \Omega^b_+)= 
	 \{\zeta\in \overline \d\colon \Im(\zeta)>0\},
\end{equation}
see \eqref{eq:zUbsiB0}, hence $\psi^b( \hat U^b\cap \hat \Omega^b_-)=\{\zeta\in \overline \d\colon\Im(\zeta)<0\}$.
 It follows from \eqref{eq:Ua} and \eqref{eq:psiOm+-} that
$\psi^b\circ J^b=-\psi^b$, $\psi^b(b)=0$, $\psi^b(\hat \Gamma^b)=[-1,1]$, and  $\psi^b(\{e^b_1,e^b_2\})=\{-1,1\}$. See Figure \ref{fig:Psi}.
\begin{figure}[ht]
	\includegraphics[width=.67\textwidth]{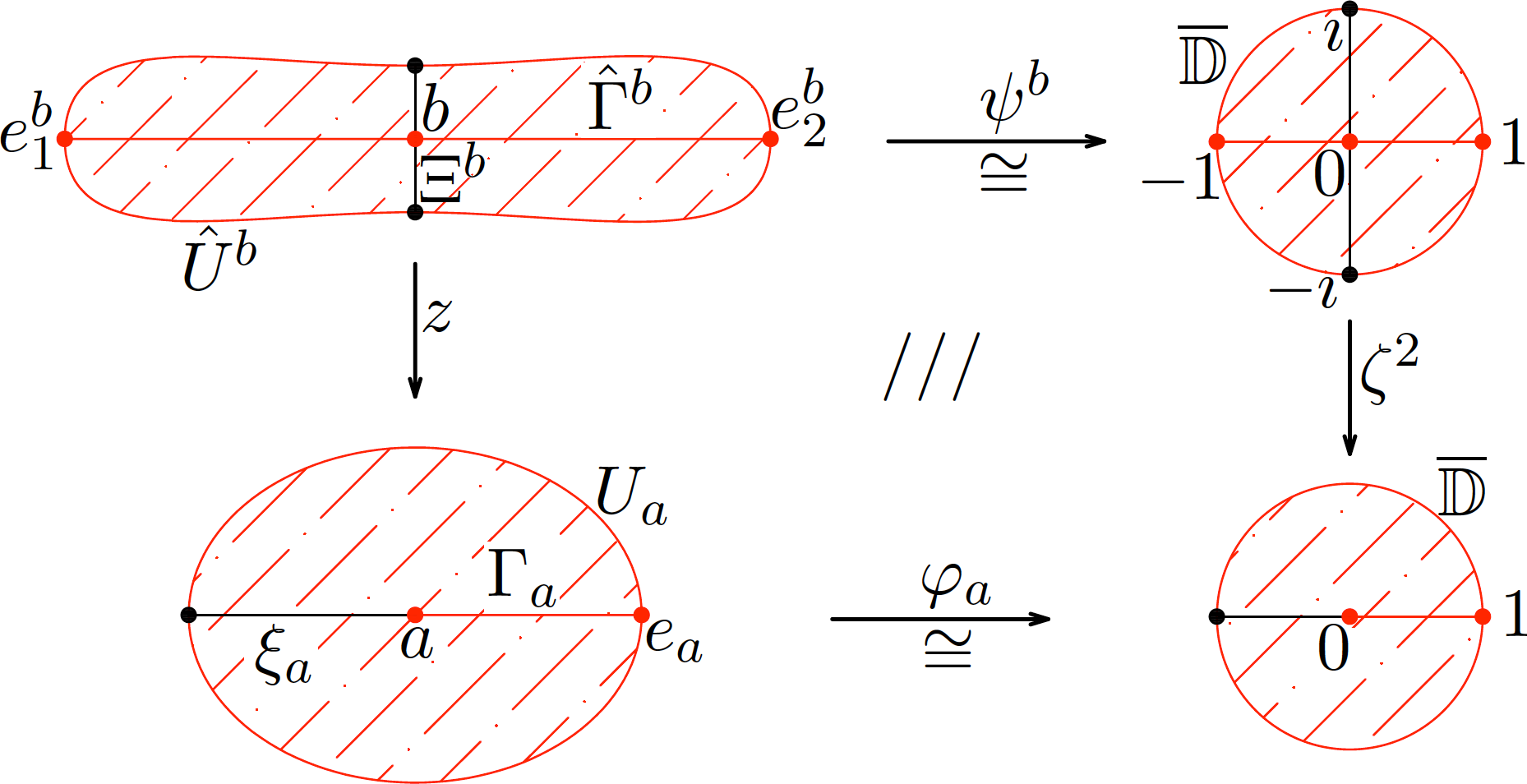}
	\caption{The biholomorphic map $\psi^b$, $b\in B_0$.}\label{fig:Psi}
\end{figure}

Choose an injective map $\lambda\colon B\setminus B_0\to (1,+\infty)$ and extend the given function $h_1\in\Re\Ocal(R)^2$ to any function $h_1\in \Re\Acal(\hat K_1)$ (see \eqref{eq:K1}) such that:
\begin{enumerate}[{\rm (A{1})}]
 \item $h_1|_{\hat\Gamma_{A_0}}$ is locally constant; see \eqref{eq:hatGammaA_0}. Thus, by \eqref{eq:Cjde} and   \eqref{eq:A0}, the map $(z,h_1)$ is injective on $\hat\Gamma_{A_0}$. 
 \smallskip
\item $h_1|_{\hat U^b}=\left\{
\begin{array}{ll}
 \displaystyle\Re(\psi^b) & \text{ for all }b\in B_0 \smallskip \\
  \lambda(b)  & \text{ for all }b\in B\setminus B_0.
  \end{array}
  \right.$
  \smallskip
\item $h_1|_{\hat \hgot}$ is injective.
 \end{enumerate}
Recall that $(R\cup \hat \Gamma_{A_0})\cap (\hat U_A\cup\hat \hgot)=\varnothing$ and $\hat U_A\cap\hat \hgot=\varnothing$. From now on, and with an obvious abuse of notation (cf.\ \eqref{eq:Cjde}), we redefine
\[
 	[ h_1 ] = \{p\in \hat K_1 \colon \exists\, q\in \hat K_1 \setminus  
	\{p\} \text{ such that }(z,h_1)(p)=(z,h_1)(q)\}.
\]
Note that {\rm (B{1})}, {\rm (C{1})}, \eqref{eq:zUbnoB0},  \eqref{eq:psiOm+-}, and the fact that $\lambda\colon B\setminus B_0\to (1,+\infty)$ is injective guarantee that
\begin{equation}\label{eq:h1z}
	[h_1]\cap\hat\hgot=\varnothing,\quad [h_1]\cap \hat U_A = \bigcup_{b\in B_0}(\psi^{b})^{-1}([-\imath,\imath]\setminus\{0\})\subset\big(\bigcup_{b\in B_0}\hat U^b\big)\setminus B_0,
\end{equation}
where $\imath=\sqrt{-1}$. 
This, {\rm (A{1})}, and the facts that $K_1\cap s\s^1=\{e_a\colon a\in A^*\}\subset s\s^1\setminus \hgot$ and that $\varphi_a(e_a)=1$ for all $a\in A$ give that
\begin{equation}\label{eq:h1zea1}
	[h_1]\cap \bb S=\varnothing.
\end{equation}

%
%
\subsubsection*{Step 3: The function $\tilde h_1$}
For any tubular neighborhood  $W$ of $K_1\setminus \hgot$ in $s \overline \d$, we let $W_a$ be the component of $W$ containing $a\in A$, $W_A=\bigcup_{a\in A}W_a$, and $W_1=W\setminus W_A$. We set  $\hat W=z^{-1}(W)$,  denote by  $\hat W^b$ the component of $\hat W$ containing $b\in B$, and call  $\hat W_A= z^{-1}(W_A)=\bigcup_{b\in B} \hat W^b$ and $\hat W_1=z^{-1}(W_1)$. Note that $J^b(\hat W^b)=\hat W^b$ for all $b\in B_0$; see \eqref{eq:J^b}.  See Figure \ref{fig:W}.
\begin{figure}[ht]
	\includegraphics[width=.47\textwidth]{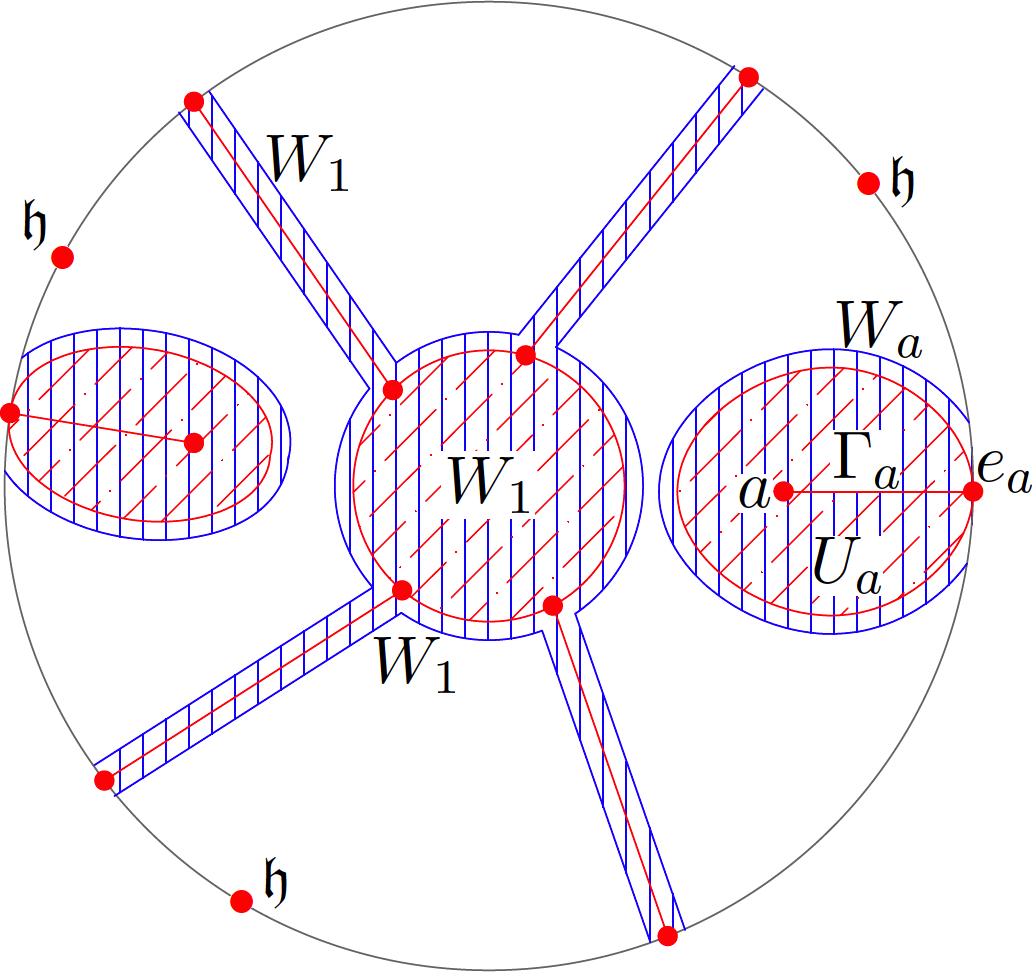}
	\caption{The set $W$.}\label{fig:W}
\end{figure}

Since $\hat K_1\subset M$ is Runge, Theorem \ref{th:Arakelyan} enables us to approximate $h_1$ by a function $\tilde h_1 \in\Re\Ocal(M)$ such that
\begin{equation}\label{eq:tildeh1-interpola}
	\text{$\tilde h_1-h_1$ vanishes to order $1$ everywhere in ${\rm Crit}(z|_S)\cup ([h_1]\cap \bb  R)$};
\end{equation} 
note that ${\rm Crit}(z|_S)\subset \hat K_1$.
Moreover, by {\rm (A{1})}, {\rm (B{1})}, {\rm (C1)}, \eqref{eq:h1z}, \eqref{eq:h1zea1}, and \eqref{eq:tildeh1-interpola}, and taking $\tilde h_1$ sufficiently close to $h_1$ on $\hat K_1$, we may assume that there is a smoothly bounded tubular neighborhood $W$ of $K_1\setminus \hgot$ in $s \overline \d$ satisfying $W\cap\hgot =\varnothing$ and the following conditions:
\begin{enumerate}[{\rm (a{1})}]
\item Defining $[\tilde h_1]=\{p\in S \colon \exists\, q\in S  \setminus  \{p\}\; 
\text{such that }(z,\tilde h_1)(p)=(z,\tilde h_1)(q)\}$, the set $[\tilde h_1]\cup B_0$
 is $1$-dimensional real analytic. Define
\[
	\Xi=([\tilde h_1]\cup B_0)\cap  \hat W_A .
\]

\smallskip
\item $[\tilde h_1]$ is  transverse to $\bb  R$ and   $[\tilde h_1]\cap \bb  R=[h_1]\cap \bb  R$; see \eqref{eq:zC1zC2}, \eqref{eq:Cjtr},  and \eqref{eq:tildeh1-interpola}.

\smallskip
\item $[\tilde h_1]\cap (\hat\Gamma_{A_0}\cup \hat\hgot )=\varnothing$. 

\smallskip
\item We have $\Xi =  \bigcup_{b\in B_0} \Xi^b$,
where for each $b\in B_0$ the set $\Xi^b$ is an analytic Jordan arc in  $\hat W^b$ meeting $\hat \Gamma^b$ only at $b$, having its endpoints in $\bb  W^b\setminus \bb S$,  being orthogonal to $\hat \Gamma^b$ at $b$ and transverse to $\bb \hat  U^b$, and meeting $\hat U^b$ in a Jordan arc; see \eqref{eq:h1z} and note that $\Xi^b=J^b(\Xi^b)$ is a slight deformation of an analytic extension of $(\psi^b)^{-1}([-\imath,\imath])$. See Figures \ref{fig:Psi} and \ref{fig:Y}.
\smallskip
\item $\big(\tilde h_1,\Im(\psi^b)\big)\colon \hat U^b\to  \r^2$ is a harmonic embedding for all $b\in B_0$; recall that $\psi^b\colon \hat U^b\to  \overline \d$ is biholomorphic.

\smallskip
\item $[\tilde h_1]\cap \hat W\cap \bb S=\varnothing$.  See \eqref{eq:h1zea1}.

\smallskip
\item $\|\tilde h_1-h_1\|_{\hat K_1}<\epsilon$.

\smallskip
\item $(z|_R,\tilde h_1|_R,h_2)\in \Ocal(R)\times\Re\Ocal(R)^2$ is an embedding. See {\rm (I)} in the lemma.
\end{enumerate} 

This concludes the construction of the function $\tilde h_1$. We now turn to the second stage of the proof: the construction of  the function $\tilde h_2\in\Re\Ocal(S)$. Again, we shall first describe a Runge compact set $\hat K_2$ in $S\setminus\mathring R$ (see \eqref{eq:K2}), extend $h_2$ to a suitable function in $\Re\Acal(\hat K_2)$, and apply uniform approximation. The set $\hat K_2$ and the function $h_2$ will be chosen so that  $[\tilde h_1]\subset \hat K_2$, $ (z, \tilde h_2)$ is injective on $[\tilde h_1]$, and $(z,\tilde h_1,\tilde h_2)$ is immersive on $B_0$.

%
%
\subsubsection*{Step 4: The second Runge compact set}
We define
\begin{equation}\label{eq:beta-a}
  	\text{$\xi_a=z(\Xi^b)$ for all $b\in B_0$ and $a=z(b)\in A$.}
\end{equation}
In view of {\rm (d1)}, it turns out that $\xi_a\subset W_a$ is an analytic Jordan arc containing $a$, otherwise disjoint from $\Gamma_a$ and tangent to $\Gamma_a$ at $a$, having endpoints $a$ and a point in $\bb W_a$, and being transverse to $\bb W_a$, $a\in A$. See Figures \ref{fig:Psi} and \ref{fig:Y}.  Set
\begin{equation}\label{alpha}
	  \xi:=z(\Xi)=\bigcup_{a\in A} \xi_a\subset W_A.
\end{equation}
 Let $Y$ be a region in $s\overline \d\setminus (r \d\cup \Gamma_{A_0}\cup U_A)$ satisfying the following conditions:
 \begin{enumerate}[11]
 \item[$\bullet$] $s\overline \d \setminus  \mathring W\subset Y$ and  $\bb W$ lies in the interior of $Y$ relative to $ s\overline \d$. So, $\hgot\subset Y$.
 
\smallskip 
\item[$\bullet$] For any component $\Omega$ of $s\overline \d\setminus (r \d\cup \Gamma_{A_0})$, $Y_\Omega:=\Omega \cap Y$ is a smooth closed disc such that $\bb  Y_\Omega \cap s\s^1\neq \varnothing$ and $Y_\Omega \cap r \s^1$ consists of a single point $y_\Omega$ such that $\Omega \cap z([\tilde h_1])\cap  r \s^1\subseteq \{y_\Omega\}$; see \eqref{eq:A0zC1} and {\rm (b1)}. In particular,     $\bb Y_\Omega$ and $r\s^1$ are tangent at $y_\Omega$. It follows from \eqref{eq:A0} and \eqref{eq:A0zC1} that
\begin{equation}\label{eq:h2yom}
	y_\Omega\in r \s^1\setminus A_0\subset r\s^1\setminus z([h_2]).
\end{equation}
For each component $\hat \Omega$ of $z^{-1}(\Omega)$, let $y_{\hat \Omega}\in \bb \hat \Omega$ be the only point
\begin{equation}\label{eq:yOmegahat}
	y_{\hat \Omega}\in z^{-1}(y_\Omega)\cap\hat \Omega.
\end{equation}

\item[$\bullet$]  $(z([\tilde h_1])\cup A) \setminus (r\d\cup Y) \subset \xi$; see {\rm (a1)}, {\rm (b1)}, {\rm (c1)}, and \eqref{alpha}.

\smallskip
\item[$\bullet$]  $\xi_a$ and $\bb Y$ meet transversally for all $a\in A$ and 
\[
	\varpi_a:=\xi_a \setminus (\mathring Y \cup \mathring U_a) \text{ is a Jordan arc};
\]
see \eqref{eq:beta-a}. 
Define 
\[
	\varpi=\bigcup_{a\in A}\varpi_a=z([\tilde h_1]) \setminus (r\overline\d\cup \mathring Y\cup\mathring U_A) 
\]
and
\begin{equation}\label{eq:varpi-varpi-2}
	\hat \varpi=z^{-1}(\varpi)=[\tilde h_1]\setminus  (R\cup \mathring{\hat Y}\cup \mathring{\hat U}_A).
\end{equation}
\end{enumerate}
See Figures \ref{fig:Psi}, \ref{fig:Y}, and \ref{fig:K2}.
\begin{figure}[ht]
\begin{minipage}[b]{0.27\linewidth}
\centering
	\includegraphics[width=\textwidth]{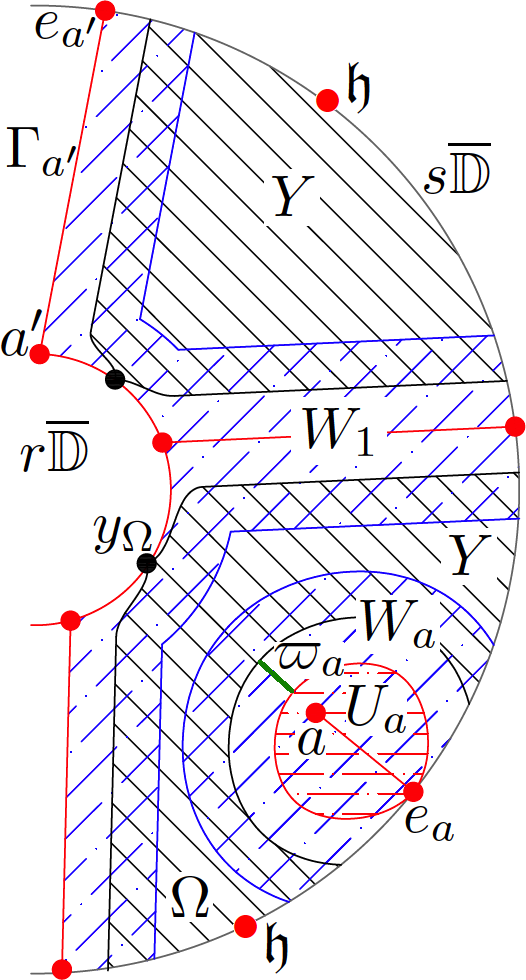}
 \end{minipage} 
 \hspace{10mm}
\begin{minipage}[b]{0.45\linewidth}
\centering
	\includegraphics[width=\textwidth]{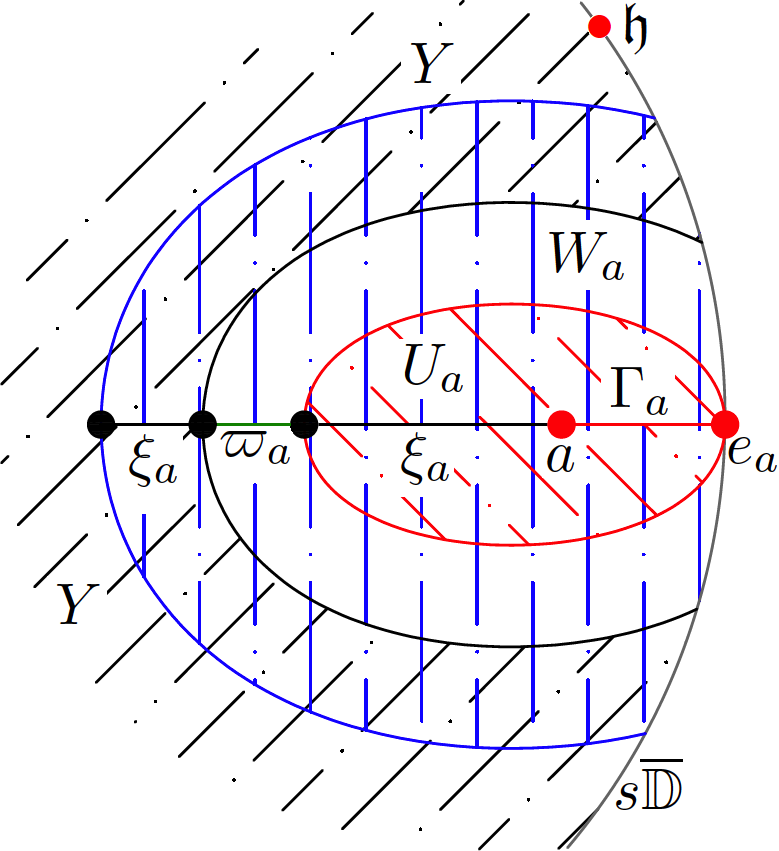}
 \end{minipage} 
	\caption{The sets $Y$ and $\varpi_a$.}\label{fig:Y}
\end{figure}
Define $\hat Y=z^{-1}(Y)$. For each $b\in B_0$ let 
\[
	\hat \varpi^b=z^{-1}(\varpi_a)\cap \Xi^b=z^{-1}(\varpi_a)\cap \hat W^b,\quad a=z(b).
\] 
Note that $\hat \varpi^b$ is the closure of $\Xi^b\setminus (\hat Y\cup \hat U^b)$. 
It turns out that $\hat \varpi^b$ consists of two  connected components of $\hat \varpi$, namely 
\begin{equation}\label{eq:Ombet+-}
	\hat \varpi^b_+\subset \hat \Omega^b_+\quad \text{and}\quad \hat \varpi^b_-=J^b(\hat \varpi^b_+)\subset \hat \Omega^b_-;
\end{equation}
see \eqref{Omega-Omega+}. 
Define
\begin{equation}\label{eq:K2}
	K_2= r\overline \d  \cup Y\cup \varpi \cup U_A \quad \text{and}\quad  \hat K_2= z^{-1}(K_2).
\end{equation}
See Figure \ref{fig:K2}.
\begin{figure}[ht]
\begin{minipage}[b]{0.56\linewidth}
\centering
	\includegraphics[width=\textwidth]{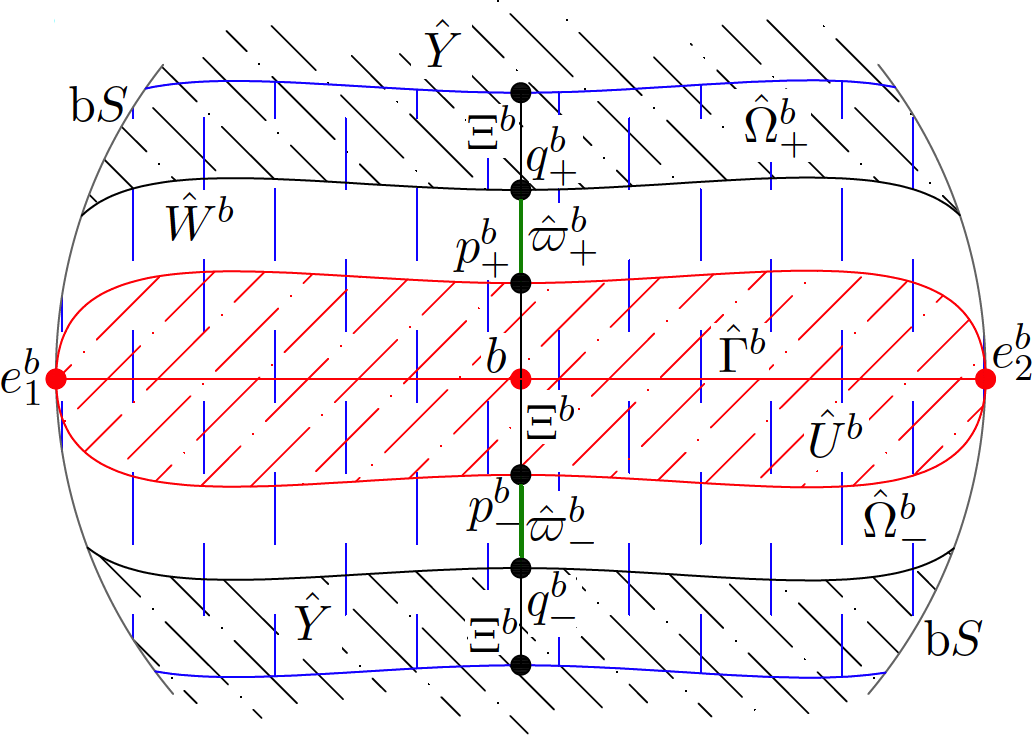}
 \end{minipage} 
   \hspace{1mm}
     \begin{minipage}[b]{0.41\textwidth}
 \centering
	\includegraphics[width=\textwidth]{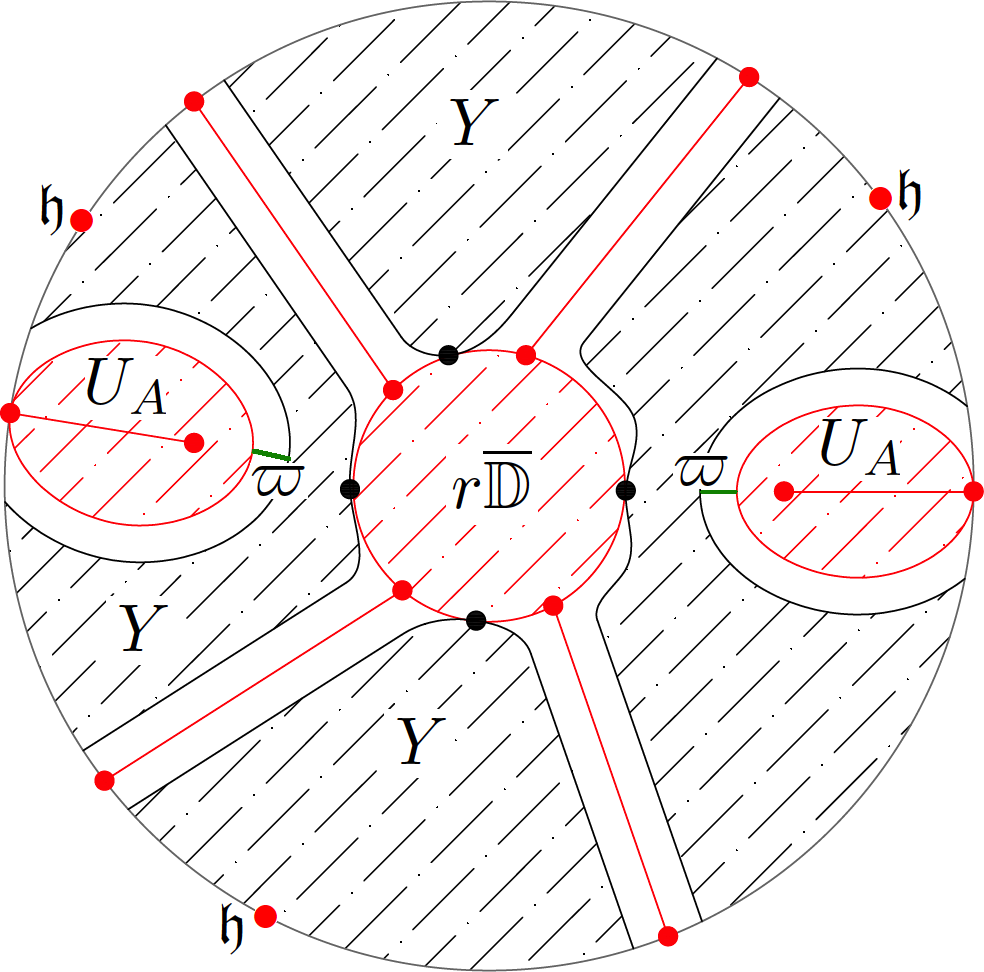}
 \end{minipage} 
	\caption{Left: The sets $\hat\varpi^b_+$ and $\hat\varpi^b_-$. Right: The set $K_2$.}\label{fig:K2}
\end{figure} 
Observe that $r\overline \d$ is a strong deformation retract of $K_2$, hence $K_2$ is Runge in $\c$ and $\hat K_2$ is Runge in $M$. By {\rm (a1)},  {\rm (d1)}, and \eqref{eq:varpi-varpi-2}, we have
\begin{equation}\label{eq:varpi}
	[\tilde h_1] \subset R\cup \hat Y\cup \hat U_A\cup\big(\bigcup_{b\in B_0}\hat \varpi^b\big) \subset \hat K_2.
\end{equation}
%
%
\subsubsection*{Step 5: Extending $h_2$ to $\hat K_2$}
We now extend $h_2$ to any function in $\Re\Acal(\hat K_2)$ satisfying the following conditions.
\begin{enumerate}[{\rm (A{2})}]
\item $h_2|_{\hat Y}$ is locally constant, and hence $(z,h_2)$ is injective on $\hat Y$; use \eqref{eq:Cjde} and  \eqref{eq:h2yom}. 
\smallskip
\item  $h_2|_{\hat U^b}=\Im(\Psi^b)$ for all $b\in B_0$, see \eqref{eq:psiOm+-}.
\smallskip
\item $(z,h_2)|_{\hat \varpi^b}\colon \hat \varpi^b\to \r^3$ is an embedding for all $b\in B_0$. 
\end{enumerate}
In order to construct such an extension of $h_2$ we proceed as follows. First, we extend $h_2$ to $R\cup \hat Y\cup (\bigcup_{b\in B_0}\hat U^b)$ as indicated in {\rm (A2)} and {\rm (B2)}, hence $h_2\in \Re\Acal(R\cup \hat Y\cup (\bigcup_{b\in B}\hat U^b))$. We then extend $h_2$ to any continuous function in $\hat K_2\setminus \bigcup_{b\in B_0}\hat \varpi^b$, so $h_2\in \Re\Acal(\hat K_2\setminus \bigcup_{b\in B_0}\hat \varpi^b)$. The crucial point is that we now may extend $h_2$ to $\hat\varpi^b=\hat \varpi^b_+\cup\varpi^b_-$, $b\in B_0$, so that $h_2$ is of class $\Re\Acal(\hat K_2)$ and satisfies {\rm (C{2})}.  Indeed, fix $b\in B_0$ and write $\{p^b_-\}=\hat \varpi^b_-\cap \hat U^b$ and $\{q^b_-\}=\hat \varpi^b_-\cap \hat Y$ for the endpoints of $\hat \varpi^b_-$, and  $\{p^b_+\}=\hat \varpi^b_+\cap \hat U^b$ and  $\{q^b_+\}=\hat \varpi^b_+\cap \hat Y$  for the ones of $\hat \varpi^b_+$. See Figure \ref{fig:K2}.
Conditions {\rm (A2)}, \eqref{eq:yOmegahat},  and \eqref{eq:Ombet+-} give that $h_2(q^b_-)=h_2(y_{\hat \Omega^b_-})$ and $h_2(q^b_+)=h_2(y_{\hat \Omega^b_+})$, and hence 
\[
	h_2(q^b_-)<h_2(q^b_+)
\] 
by \eqref{eq:OmOm'}, \eqref{Omega-Omega+}, and \eqref{eq:yOmegahat}.
Likewise,   \eqref{eq:psiOm+-}, \eqref{eq:Ombet+-}, and  {\rm (B{2})} ensure that
\[
	h_2(p^b_-)<0<h_2(p^b_+).
\] 
These two inequalities allows to continuously extend $h_2$ to $\hat\varpi^b=\hat \varpi^b_+\cup \varpi^b_-$ so that $h_2(p_-)<h_2(p_+)$ for all points $p_-\in \hat \varpi^b_-$ and $p_+\in \hat \varpi^b_+$ with $z(p_-)=z(p_+)$. This guarantees  {\rm (C{2})}.

%
%
\subsubsection*{Step 6: The function $\tilde h_2$}
Since $\hat K_2$ is Runge, Theorem \ref{th:Arakelyan} enables us to approximate $h_2$ uniformly on $\hat K_2$ by a function $\tilde h_2\in\Re\Ocal(M)$ such that
\begin{equation}\label{eq:emb-l}
	\text{$\tilde h_2-h_2$ vanishes to order $1$  at every point in ${\rm Crit}(z)\cap S$.}
\end{equation}

We claim that the map $\tilde h=(\tilde h_1,\tilde h_2)\in\Re\Ocal(M)^2$ satisfies the conclusion of Lemma \ref{le:fun} provided that $\tilde h_2$ is close enough to $h_2$ on $\hat K_2$. 
Indeed, note first that $(z,\tilde h)$ is an immersion by {\rm (h1)} (see the list of properties of $\tilde h_1$), the fact that $z$ is an immersion in $S\setminus (\mathring R\cup B_0)$, and conditions {\rm (B1)}, \eqref{eq:tildeh1-interpola}, {\rm (B2)}, and \eqref{eq:emb-l}. So, to check condition {\rm (i)} in the statement of the lemma, it suffices to see that $(z,\tilde h)\colon S\to\r^4$ is injective. By {\rm (h1)} and {\rm (A2)} we have that $(z,\tilde h)$ is injective on $R\cup\hat Y$. Likewise, {\rm (a1)}, {\rm (d1)}, {\rm (e1)}, and {\rm (B2)} give that $(z,\tilde h)$ is injective on $\hat U_A$. Since the sets $R$, $\hat Y$, and $\hat U_A$ are $z$-saturated, it remains to see that $(z,\tilde h)$ is injective on $S\setminus (R\cup\hat Y\cup \hat U_A)=z^{-1}(s\overline\d\setminus(r\overline\d\cup Y\cup U_A))$. In view of \eqref{eq:varpi}, it then suffices to check that  $(z,\tilde h)$ is injective on $\bigcup_{b\in B_0}\hat \varpi^b$, which follows from {\rm (C2)}. This shows {\rm (i)}.

Conditions {\rm (g{1})}, \eqref{eq:tildeh1-interpola}, and \eqref{eq:emb-l} give {\rm (ii)} provided that the approximation of $\tilde h_2$ to $h_2$ in $R$ is close enough. Finally, ${\rm (iii)}$ follows from {\rm (c1)}, {\rm (f1)}, {\rm (A2)}, and the facts that $\hgot\subset Y\setminus W$ and that $\tilde\hgot_j=z([h_j])\cap s\s^1\neq s\s^1$ is analytic, hence finite, for $j=1,2$. This concludes the proof of Lemma \ref{le:fun}.
 
 
This completes the proof of Theorem \ref{th:h}.
Theorem \ref{th:intro} is thus proved.


\subsection*{Acknowledgements}
Research partially supported by the State Research Agency (AEI) via the grants no.\ PID2020-117868GB-I00  and PID2023-150727NB-I00, and the ``Maria de Maeztu'' Unit of Excellence IMAG, reference CEX2020-001105-M, funded by MICIU/AEI/10.13039/501100011033 and ERDF/EU; and the Junta de Andaluc\'ia grant no. P18-FR-4049; Spain.

The authors wish to thank Franc Forstneri\v c for helpful discussions.



\begin{thebibliography}{10}

\bibitem{Ahlfors1947}
L.~V. Ahlfors.
\newblock Normalintegrale auf offenen {R}iemannschen {F}l\"{a}chen.
\newblock {\em Ann. Acad. Sci. Fennicae Ser. A. I. Math.-Phys.}, 1947(35):24,
  1947.

\bibitem{AhlforsSario1960}
L.~V. Ahlfors and L.~Sario.
\newblock {\em Riemann surfaces}.
\newblock Princeton Mathematical Series, No. 26. Princeton University Press,
  Princeton, N.J., 1960.

\bibitem{AlarconForstnericLopez2021Book}
A.~Alarc\'{o}n, F.~Forstneri\v{c}, and F.~J. L\'{o}pez.
\newblock {\em Minimal surfaces from a complex analytic viewpoint}.
\newblock Springer Monographs in Mathematics. Springer, Cham, [2021] \copyright
  2021.

\bibitem{AlarconLopez2012JDG}
A.~Alarc{\'o}n and F.~J. L{\'o}pez.
\newblock Minimal surfaces in {$\mathbb{R}^3$} properly projecting into
  {$\mathbb{R}^2$}.
\newblock {\em J. Differential Geom.}, 90(3):351--381, 2012.

\bibitem{AndristWold2014AIF}
R.~Andrist and E.~F. Wold.
\newblock Riemann surfaces in {S}tein manifolds with the density property.
\newblock {\em Ann. Inst. Fourier (Grenoble)}, 64(2):681--697, 2014.

\bibitem{BellNarasimhan1990EMS}
S.~R. Bell and R.~Narasimhan.
\newblock Proper holomorphic mappings of complex spaces.
\newblock In {\em Several complex variables, {VI}}, volume~69 of {\em
  Encyclopaedia Math. Sci.}, pages 1--38. Springer, Berlin, 1990.

\bibitem{Bishop1958PJM}
E.~Bishop.
\newblock Subalgebras of functions on a {R}iemann surface.
\newblock {\em Pacific J. Math.}, 8:29--50, 1958.

\bibitem{Bishop1961AJM}
E.~Bishop.
\newblock Mappings of partially analytic spaces.
\newblock {\em Amer. J. Math.}, 83:209--242, 1961.

\bibitem{EliashbergGromov1992AM}
Y.~Eliashberg and M.~Gromov.
\newblock Embeddings of {S}tein manifolds of dimension {$n$} into the affine
  space of dimension {$3n/2+1$}.
\newblock {\em Ann. of Math. (2)}, 136(1):123--135, 1992.

\bibitem{FornaesForstnericWold2020}
J.~E. Forn{\ae}ss, F.~Forstneri\v{c}, and E.~F. Wold.
\newblock Holomorphic approximation: the legacy of {W}eierstrass, {R}unge,
  {O}ka-{W}eil, and {M}ergelyan.
\newblock In {\em Advancements in complex analysis---from theory to practice},
  pages 133--192. Springer, Cham, [2020] \copyright 2020.

\bibitem{Forster1970CMH}
O.~Forster.
\newblock Plongements des vari\'et\'es de {S}tein.
\newblock {\em Comment. Math. Helv.}, 45:170--184, 1970.

\bibitem{Forstneric2017}
F.~Forstneri{\v{c}}.
\newblock {\em Stein manifolds and holomorphic mappings. The homotopy principle
  in complex analysis (2nd edn)}, volume~56 of {\em Ergebnisse der Mathematik
  und ihrer Grenzgebiete. 3. Folge. A Series of Modern Surveys in Mathematics}.
\newblock Springer, Berlin, 2017.

\bibitem{GreeneWu1975AIF}
R.~E. Greene and H.~Wu.
\newblock Embedding of open {R}iemannian manifolds by harmonic functions.
\newblock {\em Ann. Inst. Fourier (Grenoble)}, 25(1, vii):215--235, 1975.

\bibitem{GunningRossi2009}
R.~C. Gunning and H.~Rossi.
\newblock {\em Analytic functions of several complex variables}.
\newblock AMS Chelsea Publishing, Providence, RI, 2009.
\newblock Reprint of the 1965 original.

\bibitem{Heinz1952}
E.~Heinz.
\newblock \"{U}ber die {L}\"{o}sungen der {M}inimalfl\"{a}chengleichung.
\newblock {\em Nachr. Akad. Wiss. G\"{o}ttingen. Math.-Phys. Kl.
  Math.-Phys.-Chem. Abt.}, 1952:51--56, 1952.

\bibitem{Nevanlinna1941}
R.~Nevanlinna.
\newblock Quadratisch integrierbare {D}ifferentiale auf einer {R}iemannschen
  {M}annigfaltigkeit.
\newblock {\em Ann. Acad. Sci. Fennicae Ser. A. I. Math.-Phys.}, 1941(1):34,
  1941.

\bibitem{SchoenYau1997CIP}
R.~Schoen and S.~T. Yau.
\newblock {\em Lectures on harmonic maps}.
\newblock Conference Proceedings and Lecture Notes in Geometry and Topology,
  II. International Press, Cambridge, MA, 1997.

\bibitem{Schurmann1997MA}
J.~Sch{\"u}rmann.
\newblock Embeddings of {S}tein spaces into affine spaces of minimal dimension.
\newblock {\em Math. Ann.}, 307(3):381--399, 1997.

\bibitem{Whitney1937}
H.~Whitney.
\newblock Analytic coordinate systems and arcs in a manifold.
\newblock {\em Ann. of Math. (2)}, 38(4):809--818, 1937.

\end{thebibliography}


\medskip
\noindent Antonio Alarc\'{o}n, Francisco J. L\'opez
\newline
\noindent Departamento de Geometr\'{\i}a y Topolog\'{\i}a e Instituto de Matem\'aticas (IMAG), Universidad de Granada, Campus de Fuentenueva s/n, E--18071 Granada, Spain.
\newline
\noindent  e-mail: {\tt alarcon@ugr.es}, {\tt fjlopez@ugr.es}

\end{document}